\documentclass[journal]{IEEEtran}

\usepackage[dvipsnames]{xcolor}
\usepackage{extarrows,paralist}
\usepackage{mhchem} 
\usepackage{siunitx} 
\usepackage{graphicx} 
\usepackage{amsmath,graphicx}
\usepackage{mathtools}

\usepackage{amsthm}
\newtheorem{lem}{Lemma}
\newtheorem{theorem}{Theorem}

\newtheorem{assump}{Assumption}

\usepackage[square, comma, sort&compress, numbers]{natbib}
\usepackage{mathtools}
\usepackage{graphicx}
\usepackage{float}
\usepackage{caption}
\usepackage{color}
\usepackage{placeins}
\usepackage{float}
\usepackage{tabularx,colortbl}
\usepackage[skip=2pt,font=footnotesize]{caption}
\usepackage{amsmath,subfigure}
\usepackage{listings}
\usepackage[outdir=./]{epstopdf}
\usepackage{amssymb}
\usepackage{multirow}
\usepackage{algorithm}
\usepackage{algpseudocode}

\usepackage{enumitem}
\allowdisplaybreaks

\makeatletter
\newcommand{\vast}{\bBigg@{4}}
\newcommand{\Vast}{\bBigg@{5}}
\makeatother

\def\mbb{\mathbb}
\def\mb{\mathbf}
\def\mc{\mathcal}

\begin{document}
	\title{Distributed Subgradient Projection Algorithm over Directed Graphs: Alternate Proof} 
	\author{Ran Xin, Chenguang Xi, and Usman A. Khan
		\thanks{
			The authors are with the ECE Department at Tufts University, Medford, MA; {\texttt{chenguang.xi@tufts.edu, khan@ece.tufts.edu}}. This work has been partially supported by an NSF Career Award \# CCF-1350264.}
	}
	
	\maketitle
	\begin{abstract}
		We propose Directed-Distributed Projected Subgradient (D-DPS) to solve a constrained optimization problem over a multi-agent network, where the goal of agents is to collectively minimize the sum of locally known convex functions. Each agent in the network owns only its local objective function, constrained to a commonly known convex set. We focus on the circumstance when communications between agents are described by a \emph{directed} network. The D-DPS combines surplus consensus to overcome the asymmetry caused by the directed communication network. The analysis shows the convergence rate to be $O(\frac{\ln k}{\sqrt{k}})$.
	\end{abstract}
	
	\textcolor{blue}{{\small\textit{\textbf{Disclaimer---}}This manuscript provides an alternate approach to prove the results in \textit{C. Xi and U. A. Khan, Distributed Subgradient Projection Algorithm over Directed Graphs, in IEEE Transactions on Automatic Control}. The changes, colored in blue, result into a tighter result in Theorem~1. }}
	
\section{Introduction}\label{s1}
We focus~\cite{c_Hooi-Tong} on distributed methods to solve constrained minimization of a sum of convex functions, where each component is known only to a specific agent in a multi-agent network. The formulation has applications in, e.g., distributed sensor networks,~\cite{distributed_Rabbit}, machine learning,~\cite{distributed_Cevher,distributed_Mateos}, and low-rank matrix completion,~\cite{distributed_Ling}. Most existing algorithms assume the information exchange over undirected networks, i.e., if agent~$i$ can send information to agent~$j$, then agent~$j$ can also send information to agent~$i$. In many other realistic scenarios, however, the underlying graph may be directed. In the following, we summarize related literature on distributed optimization over multi-agent networks, which is either undirected or directed. 

\textbf{Undirected Graphs:} The corresponding problem over undirected graphs can fall into either the primal or the dual formulation, the choice of which depends on the mathematical nature of the applications. Typical primal domain methods include~\cite{uc_Nedic,cc_nedic,cc_Duchi,cc_Johansson,srivastava2011distributed,ram2010distributed}, where a convergence rate $O(\ln k/\sqrt{k})$ is obtained due to the diminishing step-size. To accelerate the rate, Ref.~\cite{fast_Gradient} applies the Nesterov-based method, achieving $O(\ln k/k^2)$ with the Lipschitz continuous gradient assumption. A related algorithm, EXTRA,~\cite{EXTRA}, uses a constant step-size and the gradients of the last two iterates. The method converges linearly under a strong-convexity assumption. The main advantage of primal domain methods is their computational simplicity. Dual domain methods formulate the problem into a constrained model: at each iteration for a fixed dual variable, the primal variables are first solved to minimize some Lagrangian-related functions, then the dual variables are updated accordingly,~\cite{dual_Terelius}. The distributed Alternating Direction Method of Multipliers (ADMM),~\cite{ADMM_Mota,ADMM_Shi,ADMM_Wei,ADMM_jakovetic,bianchi2014stochastic}, modifies traditional dual domain methods by introducing a quadratic regularization term and provides an improvement in the numerical stability. The dual domain methods, including distributed ADMM, are often fast, but comes with a high computation burden. To overcome this, Refs.~\cite{ADMM_Hong, ADMM_Ling} approximate the distributed implementation of ADMM. The computational complexity is similar to the primal domain methods. Random Coordinate Descent Methods,~\cite{Necoara_CDA0,Necoara_CDA1}, are also used in the dual formulation, which are better suited when the dimension of data is very large. 

\textbf{Directed Graphs:} Recent papers,~\cite{opdirect_Nedic,opdirect_Xi,opdirect_Makhdoumi,opdirect_Xi2}, consider distributed optimization over directed graphs. Among them, Refs.~\cite{opdirect_Nedic,opdirect_Xi,opdirect_Makhdoumi} consider non-smooth optimization problems. Subgradient-Push,~\cite{opdirect_Nedic}, applies the push-sum consensus,~\cite{ac_directed,ac_directed0}, to subgradient-based methods. Directed-Distributed Subgradient Descent,~\cite{opdirect_Xi}, is another subgradient-based alternative, combining surplus consensus,~\cite{ac_Cai1}. Ref.~\cite{opdirect_Makhdoumi} combines the weight-balancing technique,~\cite{c_Hooi-Tong}, with the subgradient-based method. These subgradient-based method,~\cite{opdirect_Nedic,opdirect_Xi,opdirect_Makhdoumi}, restricted by diminishing step-sizes, converge at~$O(\ln k/\sqrt{k})$. A recent algorithm, DEXTRA,~\cite{opdirect_Xi2}, is a combination of push-sum and EXTRA. It converges linearly under the strong-convexity assumption on the objective functions. In contrast to this work, Refs.~\cite{opdirect_Nedic,opdirect_Xi,opdirect_Makhdoumi,opdirect_Xi2} all solve unconstrained problems.

The major contribution of this paper is to provide and analyze the \emph{constrained} protocol over \emph{directed} graphs, i.e., each agent is constrained to some convex set and the communication is directed. To these aims, we provide and analyze the Directed-Distributed Projected Subgradient (D-DPS) algorithm in this paper. It is worth mentioning that generalizing existing work on unconstrained problems over undirected graphs is non-trivial because of two reasons: (i) the non-expansion property of the projection operation is not directly applicable; and, (ii) the weight matrices cannot be doubly stochastic, due to which the information exchange between two agents is asymmetric. We treat this asymmetry by bringing in ideas from surplus consensus,~\cite{ac_Cai1,opdirect_Xi}. We show that D-DPS converges at $O(\ln k/\sqrt{k})$ for non-smooth functions.

\textbf{Notation:} We use lowercase bold letters to denote vectors and uppercase italic letters to denote matrices. We denote by $[A]_{ij}$ or $a_{ij}$ the $(i,j)$th element of a matrix,~$A$. An~$n$-dimensional vector with all elements equal to one (zero) is represented by~$\mb{1}_n$ ($\mb{0}_n$). The notation~$0_{n\times n}$ represents an~$n\times n$ matrix with all elements equal to zero, and $I_{n\times n}$ the $n\times n$ identity matrix. The inner product of two vectors~$\mb{x}$ and~$\mb{y}$ is~$\langle\mb{x},\mb{y}\rangle$. We use~$\|\mb{x}\|$ to denote the standard Euclidean norm of $\mb{x}$. For a function $f(\mb{x})$, we denote its subgradient at $\mb{x}$ by $\nabla f(\mb{x})$. Finally, we use $\mc{P}_{\mc{X}}[\mb{x}]$ for the projection of a vector $\mb{x}$ on the set $\mc{X}$, i.e.,  $\mc{P}_{\mc{X}}[\mb{x}]=\arg\min_{\mb{v}\in\mc{X}}\|\mb{v}-\mb{x}\|^2$.

\section{Problem Formulation and Algorithm}\label{s2}
Consider a strongly-connected network of~$n$ agents communicating over a \emph{directed} graph,~$\mc{G}=(\mc{V},\mc{E})$, where~$\mc{V}$ is the set of agents, and~$\mc{E}$ is the collection of ordered pairs,~$(i,j),i,j\in\mc{V}$, such that agent~$j$ can send information to agent~$i$. Define~$\mc{N}_i^{{\scriptsize \mbox{in}}}$ to be the collection of in-neighbors that can send information to agent~$i$. Similarly,~$\mc{N}_i^{{\scriptsize \mbox{out}}}$ is defined as the out-neighbors of agent~$i$. We allow both~$\mc{N}_i^{{\scriptsize \mbox{in}}}$ and~$\mc{N}_i^{{\scriptsize \mbox{out}}}$ to include the node~$i$ itself. In our case, $\mc{N}_i^{{\scriptsize \mbox{in}}}\neq\mc{N}_i^{{\scriptsize \mbox{out}}}$ in general. We focus on solving a constrained convex optimization problem that is distributed over the above multi-agent network. In particular, the network of agents cooperatively solve the following optimization problem:
\begin{align}
\mbox{P1}:\quad&\mbox{minimize }\quad f(\mb{x})=\sum_{i=1}^nf_i(\mb{x}),\qquad\mbox{subject to}\quad\mb{x}\in\mc{X},\nonumber
\end{align}
where each local objective function~$f_i:\mbb{R}^p\rightarrow\mbb{R}$ being convex, not necessarily differentiable, is only known by agent~$i$, and the constrained set, $\mc{X}\subseteq\mbb{R}^p$, is convex and closed.

The goal is to solve problem P1 in a distributed manner such that the agents do not exchange the objective function with each other, but only share their own states with their out-neighbors in each iteration. We adopt the following standard assumptions.
\begin{assump}\label{asp1}
The graph $\mc{G}=(\mc{V},\mc{E})$ is strongly-connected, i.e.,~$\forall i, j\in\mc{V}$, there exists a directed path from $j$ to $i$. 
\end{assump}
\noindent Assumption \ref{asp1} ensures that the information from all agents is disseminated to the whole network such that a consensus can be reached. For example, a directed spanning tree does not satisfy Assumption \ref{asp1} as the root of this tree cannot receive information from any other agent.
\begin{assump}\label{asp2}
	Each function,~$f_i$, is convex, but not necessarily differentiable. The subgradient,~$\nabla f_i(\mb{x})$, is bounded, i.e.,~$\|\nabla f_i(\mb{x})\|\leq B_{f_i}$, $\forall\mb{x}\in\mbb{R}^p$. With~$B=\max_i\{{B_{f_i}}\}$, we have for any~$\mb{x}\in\mbb{R}^{p}$,
	\begin{align}\label{grad}
	\left\|\nabla f_i(\mb{x})\right\|&\leq B,\qquad\forall i\in\mc{V}.
	\end{align}
\end{assump}
\begin{color}{blue}
	\begin{assump} \label{asp3}
		The optimal solution set, denoted by $\mc{X}^*$, is non-empty.
	\end{assump}
\end{color}

\subsection{Algorithm: D-DPS}
Let each agent,~$j\in\mc{V}$, maintain two vectors:~$\mb{x}_j^k$ and~$\mb{y}_j^k$, both in~$\mbb{R}^p$, where~$k$ is the discrete-time index. At the~$k+1$th iteration, agent~$j$ sends its state estimate,~$\mb{x}_j^k$, as well as a weighted auxiliary variable,~$b_{ij}\mb{y}_j^k$, to each out-neighbor\footnote{To implement this, each agent $j$ only need to know its out-degree, and set $b_{ij}=1/|\mc{N}_j^{{\scriptsize \mbox{out}}}|$. This assumption is standard in the related literature regarding distributed optimization over directed graphs,~\cite{opdirect_Nedic,opdirect_Xi,opdirect_Makhdoumi,opdirect_Xi2}}, $i\in\mc{N}_j^{{\scriptsize \mbox{out}}}$, where all those out-weights,~$b_{ij}$'s, of agent $j$ satisfy:

\begin{equation*}
b_{ij}=\left\{
\begin{array}{rl}
>0,&i\in\mc{N}_j^{{\scriptsize \mbox{out}}},\\
0,&\mbox{otw.},
\end{array}
\right.
\qquad
\sum_{i=1}^nb_{ij}=1.
\end{equation*}
Agent~$i$ then updates the variables,~$\mb{x}_i^{k+1}$ and~$\mb{y}_i^{k+1}$, with the information received from its in-neighbors,~$j\in\mc{N}_i^{{\scriptsize \mbox{in}}}$:
\begin{subequations}\label{alg1}
	\begin{align}
	\mb{x}_i^{k+1}&=\mc{P}_{\mc{X}}\left[\sum_{j=1}^na_{ij}\mb{x}_j^k+\epsilon\mb{y}_i^k-\alpha_k\nabla\mb{f}_i^k\right],\label{alg1a}\\
	\mb{y}_i^{k+1}&=\mb{x}_i^k-\sum_{j=1}^na_{ij}\mb{x}_j^k+\sum_{j=1}^n\left(b_{ij}\mb{y}_j^k\right)-\epsilon\mb{y}_i^k,\label{alg1b}
	\end{align}
\end{subequations}
where the in-weights,~$a_{ij}$'s, of agent $i$ satisfy that:
\begin{equation*}
a_{ij}=\left\{
\begin{array}{rl}
>0,&j\in\mc{N}_i^{{\scriptsize \mbox{in}}},\\
0,&\mbox{otw.},
\end{array}
\right.
\qquad
\sum_{j=1}^na_{ij}=1;
\end{equation*}
The scalar,~$\epsilon$, is a small positive constant, of which we will give the range later. The diminishing step-size,~$\alpha_k\geq0$, satisfies the persistence conditions:
$\sum_{k=0}^\infty\alpha_k=\infty;\sum_{k=0}^\infty\alpha_k^2<\infty;$ 
\begin{color}{blue}
	we also require $\alpha_k$ to be non-increasing, see e.g.,~\cite{opdirect_Nedic},
\end{color}and $\nabla \mb{f}_i^k=\nabla f_i(\mb{x}_i^k)$ represents the subgradient  of $f_i$ at $\mb{x}_i^k$. We provide the proof of D-DPS in Section \ref{s3}, where we show that all agents states converge to some common accumulation state, and the accumulation state converges to the optimal solution of the problem, i.e., $\mb{x}_i^\infty=\mb{x}_j^\infty=\mb{x}^\infty$ and $f(\mb{x}^\infty)=f^*$, $\forall i, j$, where $f^*$ denotes the optimal solution of Problem P1. To facilitate the proof, we present some existing results regarding the convergence of a new weighting matrix, and some inequality satisfied by the projection operator.
\subsection{Preliminaries}
Let $A=\left\{a_{ij}\right\}\in\mbb{R}^{n\times n}$ be some row-stochastic weighting matrix representing the underlying graph $\mc{G}$, and $B=\left\{b_{ij}\right\}\in\mbb{R}^{n\times n}$ be some column-stochastic weighting matrix regarding the same graph $\mc{G}$. Define $M\in\mbb{R}^{2n\times 2n}$ the matrix as follow.
\begin{align}
M&=\left[
\begin{array}{cc}
A & \epsilon I_{n\times n} \\
I_{n\times n}-A & B-\epsilon I_{n\times n} \\
\end{array}
\right],\label{M}
\end{align}
where $\epsilon$ is some arbitrary constant. We next state an existing result from our prior work,~\cite{opdirect_Xi} (Lemma 3), on the convergence performance of $M^\infty$. The convergence of $M$ is originally studied in~\cite{ac_Cai1}, while we show the geometric convergence in~\cite{opdirect_Xi}. Such a matrix $M$ is crucial in the convergence analysis of D-DPS provided in Section \ref{s3}.
\begin{lem}\label{lem_M2}
	Let Assumption \ref{asp1} holds. Let~$M$ be the weighting matrix,~Eq.~\eqref{M}, and the constant~$\epsilon$ in~$M$ satisfy~$\epsilon\in(0,\Upsilon)$, where~$\Upsilon:=\frac{1}{(20+8n)^n}(1-|\lambda_3|)^n$ and~$\lambda_3$ is the third largest eigenvalue of~$M$ by setting~$\epsilon=0$. Then:
	
	\begin{enumerate}[label=(\alph*)]
		\item The sequence of~$\left\{M^k\right\}$, as~$k$ goes to infinity, converges to the following limit:
		\begin{align}
		\lim_{k\rightarrow\infty}M^k=\left[
		\begin{array}{cc}
		\frac{\mb{1}_n\mb{1}_n^\top}{n} & \frac{\mb{1}_n\mb{1}_n^\top}{n} \\
		0 & 0 \\
		\end{array}
		\right];\nonumber
		\end{align}
		\item For all~$i,j\in[1,\ldots,2n]$, the entries~$\left[M^k\right]_{ij}$ converge at a geometric rate, i.e., there exist bounded constants,~$\Gamma\in\mbb{R^+}$, and~$\gamma\in(0,1)$, such that
		\begin{align}
		\left\|M^k-\left[
		\begin{array}{cc}
		\frac{\mb{1}_n\mb{1}_n^\top}{n} & \frac{\mb{1}_n\mb{1}_n^\top}{n} \\
		0 & 0 \\
		\end{array}
		\right]\right\|_{\infty}\leq\Gamma\gamma^k.\nonumber
		\end{align}
	\end{enumerate}
\end{lem}
The proof and related discussion can be found in~\cite{opdirect_Xi,ac_Cai1}. The next lemma regarding the projection operator is from~\cite{cc_nedic}.
\begin{lem}\label{lem_NonexpanBregman}
	Let~$\mc{X}$ be a non-empty closed convex set in~$\mbb{R}^p$. For any vector~$\mb{y}\in\mc{X}$ and~$\mb{x}\in\mbb{R}^p$, it satisfies:
	\begin{enumerate}[label=(\alph*)]
		\item~$\left\langle\mb{y}-\mc{P}_{\mc{X}}\left[\mb{x}\right],\mb{x}-\mc{P}_{\mc{X}}\left[\mb{x}\right]\right\rangle\leq 0$.
		\item~$\left\|\mc{P}_{\mc{X}}\left[\mb{x}\right]-\mb{y}\right\|^2\leq\left\|\mb{x}-\mb{y}\right\|^2-\left\|\mc{P}_{\mc{X}}\left[\mb{x}\right]-\mb{x}\right\|^2$.\nonumber
	\end{enumerate}
\end{lem}
\section{Convergence Analysis}\label{s3}
To analyze D-DPS, we write Eq.~\eqref{alg1} in a compact form. We denote~$\mb{z}_i^k\in\mbb{R}^p$,~$\mb{g}_i^k\in\mbb{R}^p$ as
\begin{align}
\mb{z}_i^k &= \left\{
\begin{array}{l r}
\mb{x}_i^k, ~~~~~~~~~~~~~ 1\leq i\leq n,&\\
\mb{y}_{i-n}^k, ~~~~ n+1\leq i\leq2n,&
\end{array}
\right.\notag\\
\mb{g}_i^k &=
\left\{
\begin{array}{l r}
\mb{x}_i^{k+1}-\sum\limits_{j=1}^na_{ij}\mb{x}_j^k-\epsilon\mb{y}_i^k,~~~~1\leq i\leq n,&\\
\mb{0}_p,~~~~~~~~~~~~~~~~~~~~~~~~~~ n+1\leq i\leq2n,&\\
\end{array} \right.\label{g}
\end{align}
and $A=\{a_{ij}\},B=\{b_{ij}\}$, and $M=\{m_{ij}\}$ collect the weights from Eqs.~\eqref{alg1} and~\eqref{M}. We now represent Eq.~\eqref{alg1} as follows: for any~$i\in\{1,...,2n\}$, at~$k+1$th iteration,
\begin{align}\label{alg2}
\mb{z}_i^{k+1}=\sum_{j=1}^{2n}m_{ij}\mb{z}_j^{k}+\mb{g}_i^k,
\end{align}
where we refer to~$\mb{g}_i^k$ as the \textit{perturbation}. Eq.~\eqref{alg2} can be viewed as a distributed subgradient method,~\cite{uc_Nedic}, where the doubly stochastic matrix is substituted with the new weighting matrix,~$M$, Eq.~\eqref{M}, and the subgradient is replaced by the perturbation,~$\mb{g}_i^k$. We summarize the spirit of the upcoming convergence proof, which consists of proving both the consensus property and the optimality property of D-DPS. As to the consensus property, we show that the disagreement between estimates of agents goes to zero, i.e., $\lim_{k\rightarrow\infty}\|\mb{x}_i^k-\mb{x}_j^k\|=0$, $\forall i, j\in\mc{V}$. More specifically, we show that the limit of agent estimates converge to some accumulation state, $\overline{\mb{z}}^k=\frac{1}{n}\sum_{i=1}^{2n}\mb{z}_i^k$, i.e., $\lim_{k\rightarrow\infty}\|\mb{x}_i^k-\overline{\mb{z}}^k\|=0$, $\forall i$, and the agents additional variables go to zero, i.e., $\lim_{k\rightarrow\infty}\|\mb{y}_i^k\|=0$, $\forall i$. Based on the consensus property, we next show the optimality property that the difference between the objective function evaluated at the accumulation state and the optimal solution goes to zero, i.e., $\lim_{k\rightarrow\infty}f(\overline{\mb{z}}^k)=f^*$.

We formally define the accumulation state $\overline{\mb{z}}^k$ as follows,
\begin{align}\label{z}
\overline{\mb{z}}^k=\frac{1}{n}\sum_{i=1}^{2n}\mb{z}_i^k=\frac{1}{n}\sum_{i=1}^{n}\mb{x}_i^k+\frac{1}{n}\sum_{i=1}^{n}\mb{y}_i^k.
\end{align}
The following lemma regarding~$\mb{x}_i^k$,~$\mb{y}_i^k$, and~$\overline{\mb{z}}^k$ is straightforward. We assume that all of the initial states of agents are zero, i.e.,~$\mb{z}_i^0=\mb{0}_p$,~$\forall i$, for the sake of simplicity in the representation of proof. 
\begin{lem}\label{lem_consensus}
	Let Assumptions \ref{asp1}, \ref{asp2} hold. Then, there exist some bounded constants,~$\Gamma>0$ and~$0<\gamma<1$, such that:
	\begin{enumerate}[label=(\alph*)]
		\item for all $i\in\mc{V}$ and $k\geq 0$, the agent estimate satisfies\footnote{In this paper, we allow the notation that the superscript of sum being smaller than its subscript. In particular, for any sequence $\{\mb{s}_k\}$, we have $\sum_{k=k_1}^{k_2}\mb{s}_k=0$, if $k_2<k_1$. Besides, we denote in this paper for convenience that $\mb{g}_i^{-1}=\mb{0}_p$, $\forall i$}
		\begin{align}
		\left\|\mb{x}_i^k-\overline{\mb{z}}^k\right\|\leq&\Gamma\sum_{r=1}^{k-1}\gamma^{k-r}\sum_{j=1}^n\left\|\mb{g}_j^{r-1}\right\|+\sum_{j=1}^n\left\|\mb{g}_j^{k-1}\right\|;\nonumber
		\end{align}
		
		\item for all $i\in\mc{V}$ and $k\geq 0$, the additional variable satisfies
		\begin{align}
		\left\|\mb{y}_i^k\right\|\leq\Gamma\sum_{r=1}^{k-1}\gamma^{k-r}\sum_{j=1}^n\left\|\mb{g}_j^{r-1}\right\|.\nonumber
		\end{align}
	\end{enumerate}
\end{lem}
\begin{proof}
	For any~$k\geq 0$, we write Eq.~\eqref{alg2} recursively
	\begin{align}\label{eq1_lem_consensus}
	\mb{z}_i^{k}=\sum_{r=1}^{k-1}\sum_{j=1}^{n}[M^{k-r}]_{ij}\mb{g}_j^{r-1}+\mb{g}_i^{k-1}.
	\end{align}	
	We have $\sum_{i=1}^{2n}[M^{k}]_{ij}=1$ for any~$k\geq0$ since each column of~$M$ sums up to one. Considering the recursive relation of~$\mb{z}_i^{k}$ in Eq.~\eqref{eq1_lem_consensus}, we obtain that~$\overline{\mb{z}}^k$ can be written as
	\begin{align}\label{eq2_lem_consensus}
	\overline{\mb{z}}^k&=\sum_{r=1}^{k-1}\sum_{j=1}^{n}\frac{1}{n}\mb{g}_j^{r-1}+\frac{1}{n}\sum_{i=1}^{n}\mb{g}_i^{k-1}.
	\end{align}	
	Subtracting Eq.~\eqref{eq2_lem_consensus} from~\eqref{eq1_lem_consensus} and taking the norm, we obtain 
	\begin{align}\label{eq3_lem_consensus}
	\left\|\mb{z}_i^{k}-\overline{\mb{z}}^k\right\|\leq&\sum_{r=1}^{k-1}\sum_{j=1}^{n}\left\|[M^{k-r}]_{ij}-\frac{1}{n}\right\|\left\|\mb{g}_j^{r-1}\right\|\nonumber\\
	&+\frac{n-1}{n}\left\|\mb{g}_i^{k-1}\right\|+\frac{1}{n}\sum_{j\neq i}\left\|\mb{g}_j^{k-1}\right\|.
	\end{align}
	The proof of part (a) follows by applying Lemma~\ref{lem_M2} to Eq.~\eqref{eq3_lem_consensus} for~$1\leq i\leq n$, whereas the proof of part (b) follows by applying Lemma~\ref{lem_M2} to Eq.~\eqref{eq1_lem_consensus} for~$n+1\leq i\leq 2n$.
\end{proof}

\subsection{Perturbation bounds}
\textcolor{blue}{We now analyze the perturbation term,~$\mb{g}_i^k$, in the next lemmas.}
\begin{lem}\label{lem_error}
	Let Assumptions \ref{asp1}, \ref{asp2} hold. Let $\epsilon$ be the small constant used in the algorithm, Eq.~\eqref{alg1}, such that $\epsilon\leq\frac{1-\gamma}{2n\Gamma\gamma}$. Define the variable $g_k=\sum_{i=1}^n\|\mb{g}_i^k\|$. Then there exists some bounded constant $D>0$ such that for all~$K\geq 2$, ~$g_k$ satisfies:
		\begin{align}\label{gk1}
		\sum_{k=0}^K\alpha_kg_k\leq D\sum_{k=0}^K\alpha_k^2,
		\end{align}
where~$\alpha_k$ is the diminishing step-size used in the algorithm.	
\end{lem}
\begin{proof}
	Based on the result of Lemma~\ref{lem_NonexpanBregman}(b), we have 
	\begin{align}
	&\left\|\mc{P}_{\mc{X}}\left[\sum_{j=1}^na_{ij}\mb{x}_j^k+\epsilon\mb{y}_i^k-\alpha_k\nabla \mb{f}_i^k\right]-\sum_{j=1}^na_{ij}\mb{x}_j^k\right\|\nonumber\\
	&\leq\left\|\epsilon\mb{y}_i^k-\alpha_k\nabla \mb{f}_i^k\right\|.
	\end{align}
	Therefore, we obtain
	\begin{align}
	\left\|\mb{g}_i^k\right\|&\leq\left\|\mb{x}_i^{k+1}-\sum\limits_{j=1}^na_{ij}\mb{x}_j^k\right\|+\epsilon\left\|\mb{y}_i^k\right\|\nonumber,\\
	&\leq\left\|\epsilon\mb{y}_i^k-\alpha_k\nabla \mb{f}_i^k\right\|+\epsilon\left\|\mb{y}_i^k\right\|\nonumber,\\
	&\leq B\alpha_k+2\epsilon\left\|\mb{y}_i^k\right\|,
	\end{align}
	where in the last inequality, we use the relation $\|\nabla \mb{f}_i^k\|\leq B$. Applying the result of Lemma~\ref{lem_consensus}(b) regarding $\|\mb{y}_i^k\|$ to the preceding relation, we have for all $i$,
	\begin{align}
	\left\|\mb{g}_i^k\right\|\leq B\alpha_k+2\epsilon\Gamma\sum_{r=1}^{k-1}\gamma^{k-r}\sum_{j=1}^n\left\|\mb{g}_j^{r-1}\right\|.\nonumber
	\end{align}
	By defining $g_k=\sum_{i=1}^n\|\mb{g}_i^k\|$, and summing the above relation over $i$, it follows that
	\begin{align}\label{gk0}
	g_k\leq nB\alpha_k+2n\epsilon\Gamma\sum_{r=1}^{k-1}\gamma^{k-r}g_{r-1}.
	\end{align}
\begin{color}{blue}
	Multiplying both sides of the inequality above by $\alpha_k$, we have:
    \begin{equation}
    	\alpha_kg_k\leq nB\alpha_k^2+2n\epsilon\Gamma\alpha_k\sum_{r=1}^{k-1}\gamma^{k-r}g_{r-1}. \nonumber
    \end{equation}
\end{color}Summing the inequality above over time from $k=0$ to $K$, we obtain
	\begin{align}
	\sum_{k=0}^K\alpha_kg_k&\leq nB\sum_{k=0}^K\alpha_k^2+2n\epsilon\Gamma\sum_{k=0}^K\alpha_k\sum_{r=1}^{k-1}\gamma^{k-r}g_{r-1}.\nonumber
	\end{align}
\begin{color}{blue}
	Since  the step-size is decreasing, it follows that
\end{color}
	\begin{align}
	\sum_{k=0}^K\alpha_k\sum_{r=1}^{k-1}\gamma^{k-r}g_{r-1}&\leq\sum_{k=0}^K\sum_{r=1}^{k-1}\gamma^{k-r}\alpha_{r-1}g_{r-1},\nonumber\\
	&\leq\frac{\gamma(1-\gamma^{K-2})}{1-\gamma}\sum_{k=0}^{K-2}\alpha_kg_k.\nonumber
	\end{align}
		Therefore, it satisfies, for any $K\geq2$, that
		\begin{align}
		\left(1-\frac{2n\epsilon\Gamma\gamma}{1-\gamma}\right)\sum_{k=0}^K\alpha_kg_k&\leq nB\sum_{k=0}^K\alpha_k^2,\nonumber
		\end{align}
		Since $\epsilon$ can be arbitrary small, (see Lemma~\ref{lem_M2}), it is achievable that $\epsilon\leq\frac{1-\gamma}{2n\Gamma\gamma}$, which obtains the desired result.
\end{proof}
\begin{color}{blue}
Based on the result of Lemma \ref{lem_error}, we show that the perturbation, $\mb{g}_i^k$, goes to zero by presenting that there exists some constant $C$ such that $\sum_{k=0}^{K}g_k^2\leq C\sum_{k=0}^{K}\alpha_k^2$.
\end{color}
\begin{color}{blue}
\begin{lem} \label{new1}
Let Assumptions \ref{asp1}, \ref{asp2} hold. Let $\epsilon$ be the small constant used in the algorithm, Eq.~\eqref{alg1}, such that $\epsilon\leq\frac{1-\gamma}{2n\Gamma\gamma}$. Define the variable $g_k=\sum_{i=1}^n\|\mb{g}_i^k\|$. Then there exists some constants $C>0$ such that for all $K\geq0$, 
\begin{equation}
\sum_{k=0}^{K}g_k^2\leq C\sum_{k=0}^{K}\alpha_k^2.
\end{equation}
\end{lem}

\begin{proof}
	According to Eq.~\eqref{gk0}:
	\begin{align}
	g_k^2 &\leq g_k \left(nB\alpha_k+2n\epsilon\Gamma\sum_{r=1}^{k-1}\gamma^{k-r}g_{r-1}\right),\nonumber\\ 
	&\leq nB\alpha_k g_k + 2n\epsilon\Gamma g_k\sum_{r=1}^{k-1}\gamma^{k-r} g_{r-1}.\nonumber
	\end{align}
Summing the inequality above over $k$ from $0$ to $K$, we have the following:
\begin{align}
\sum_{k=0}^{K} g_k^2 
\leq&~ 
nB\sum_{k=0}^{K}\alpha_kg_k + 2n\epsilon\Gamma\sum_{k=0}^{K}g_k\sum_{r=1}^{k-1}\gamma^{k-r}g_{r-1}\nonumber,\\
\leq&~ nBD\sum_{k=0}^{K}\alpha_k^2+n\epsilon\Gamma\sum_{k=0}^{K}\sum_{r=1}^{k-1}\gamma^{k-r}\left(g_k^2+g_{r-1}^2\right),\nonumber\\
=&~nBD\sum_{k=0}^{K}\alpha_k^2+n\epsilon\Gamma\sum_{k=0}^{K}g_k^2\sum_{r=1}^{k-1}\gamma^{k-r}\nonumber\\
&+n\epsilon\Gamma\sum_{k=0}^{K}\sum_{r=1}^{k-1}\gamma^{k-r}g_{r-1}^2,\nonumber\\
\leq&~
nBD\sum_{k=0}^{K}\alpha_k^2+\frac{n\epsilon\Gamma\gamma}{1-\gamma}\sum_{k=0}^{K}g_k^2\nonumber\\
&+\frac{n\epsilon\Gamma\gamma}{1-\gamma}\sum_{k=0}^{K}g_k^2.\nonumber
\end{align} 
We now have the following equation that completes the proof:
\begin{equation}
\left(1-\frac{2n\epsilon\Gamma\gamma}{1-\gamma}\right)\sum_{k=0}^{K}\alpha_k^2 \leq nBD\sum_{k=0}^{K}\alpha_k^2 \nonumber
\end{equation}
i.e.,
\begin{equation}
\sum_{k=0}^{K}g_k^2\leq C\sum_{k=0}^{K}\alpha_k^2,\nonumber
\end{equation}
where
$C=\frac{nBD\left(1-\gamma\right)}{1-\left(1+2n\epsilon\Gamma\right)\gamma}$. \nonumber
\end{proof}
\end{color}

\subsection{Consensus in Estimates}
In Lemma \ref{lem_consensus}, we bound the disagreement between estimates of agent and the accumulation state, $\|\mb{x}_i^k-\overline{\mb{z}}^k\|$, in terms of the perturbation norm, $\sum_{j=1}^n\|\mb{g}_j^k\|$. In Lemmas \ref{lem_error} and \ref{new1}, we bound the perturbation. By combining these results, we show the consensus property of the algorithm.

\begin{lem}\label{lem_consensus2}
	Let Assumptions \ref{asp1}, \ref{asp2} hold. Let~$\left\{\mb{z}_i^k\right\}$ be the sequence over~$k$ generated by Eq.~\eqref{alg2}. Then, $\forall $$i\in\mc{V}$:
	\begin{enumerate}[label=(\alph*)]
		\item  the agents reach consensus, i.e.,~$\lim_{k\rightarrow\infty}\left\|\mb{x}_i^k-\overline{\mb{z}}^k\right\|=0;$
		\item at each agent, $\lim_{k\rightarrow\infty}\left\|\mb{y}_i^k\right\|=0.$
	\end{enumerate}
\end{lem}
\begin{color}{blue}
	\begin{proof}
From Lemma 3, we have the following:
		\begin{align}
		\left\|\mb{x}_i^k-\overline{\mb{z}}^k\right\|
		&\leq\Gamma\sum_{r=1}^{k-1}\gamma^{k-r}\sum_{j=1}^n\left\|\mb{g}_j^{r-1}\right\|+\sum_{j=1}^n\left\|\mb{g}_j^{k-1}\right\|, \nonumber \\
		&=\Gamma\sum_{r=1}^{k-1}\gamma^{k-r}g_{r-1} + g_{k-1}. \nonumber
		\end{align}
We further have
		\begin{align}
			\left\|\mb{y}_i^k\right\|
			&\leq\Gamma\sum_{r=1}^{k-1}\gamma^{k-r}\sum_{j=1}^n\left\|\mb{g}_j^{r-1}\right\|, \nonumber\\
			&=\Gamma\sum_{r=1}^{k-1}\gamma^{k-r}g_{r-1}. \nonumber
		\end{align}
		Since $\sum_{k=0}^{\infty}g_k^2\leq\infty,$ $\lim\limits_{k\to\infty}g_k=0.$ By recalling Lemma 7 in [6], we have:
		\begin{equation}
		\lim\limits_{k\to\infty}\sum_{r=1}^{k-1}\gamma^{k-r}g_{r-1}=0. 
		\end{equation}
		Therefore, we have $\lim\limits_{k\to\infty}\left\|\mb{x}_i^k-\overline{\mb{z}}^k\right\|=0$ and $\lim\limits_{k\to\infty}\left\|\mb{y}_i^k\right\|=0$, which shows that consensus over the network is achieved.
	\end{proof}
\end{color}

\subsection{Optimality }
The result of Lemma \ref{lem_consensus2} reveals the fact that all agents reach consensus. We next show that the accumulation state converges to the optimal solution of the problem.

\begin{color}{blue}
\begin{lem} \label{new}
     	Let Assumptions \ref{asp1}, \ref{asp2} hold. Let~$\left\{\mb{z}_i^k\right\}$ be the sequence over~$k$ generated by Eq.~\eqref{alg2}. For all $K\geq0$, we have the following:
     	\begin{equation} \label{alphaconsensus}
     		\sum_{k=0}^K\alpha_k\left\|\mb{x}_i^k-\overline{\mb{z}}^k\right\|
     		\leq
     		D\left(1+\frac{\Gamma}{1-\gamma}\right)\sum_{k=0}^{K}\alpha_k^2.
     	\end{equation}
     	\begin{proof}
     		Considering Lemma \ref{lem_consensus}(a), we have for any $K>0$
     		\begin{align}
     		\sum_{k=0}^K\alpha_k\left\|\mb{x}_i^k-\overline{\mb{z}}^k\right\|\leq&\Gamma\sum_{k=0}^K\sum_{r=1}^{k-1}\gamma^{k-r}\alpha_k\sum_{j=1}^n\left\|\mb{g}_j^{r-1}\right\|\nonumber\\&+\sum_{k=0}^K\alpha_k\sum_{j=1}^n\left\|\mb{g}_j^{k-1}\right\|\nonumber,\nonumber\\
     		\leq&\Gamma\sum_{k=0}^{K}\sum_{r=1}^{k-1}\gamma^{k-r}\alpha_{r-1}g_{r-1}\nonumber\\
     		&+\sum_{k=0}^{K}g_{k-1}\alpha_k \nonumber,\\
     		\leq&\left(1+\frac{\Gamma}{1-\gamma}\right)\sum_{k=0}^{K}\alpha_kg_k\nonumber,\\
     		\leq&D\left(1+\frac{\Gamma}{1-\gamma}\right)\sum_{k=0}^{K}\alpha_k^2 \nonumber,
     		\end{align}
            which completes the proof.
     	\end{proof}
\end{lem}
\end{color}

\begin{color}{blue}
	\begin{lem} \label{new2}
		Let Assumptions \ref{asp1}, \ref{asp2} hold. Define the variable $g_k=\sum_{i=1}^n\|\mb{g}_i^k\|$. Then there exists some constants $F>0$ such that for all $K\geq0$, 
		\begin{equation}
		\sum_{k=0}^{K}g_k\left\|\overline{\mb{z}}^{k+1}-\mb{x}_i^{k+1} \right\|
		\leq F\sum_{k=0}^{K}\alpha_k^2
		\end{equation}
		\begin{proof}
From Lemma \ref{lem_consensus}, we have the following:
			\begin{align}
			\left\|\mb{x}_i^{k+1}-\overline{\mb{z}}^{k+1}\right\| 
			\leq \sum_{r=1}^{k}\gamma^{k+1-r}g_{r-1}+g_k.  \nonumber
			\end{align}
			It can be derived from the above inequality that:
			\begin{align}
			&\sum_{k=0}^{K}g_k\left\|\mb{x}_i^{k+1}-\overline{\mb{z}}^{k+1}\right\|\nonumber
			\leq\Gamma\sum_{k=0}^{K}g_k\sum_{r=1}^{k}\gamma^{k+1-r}g_{r-1}+\sum_{k=0}^{K}g_k^2
			\nonumber,\\
			&\leq\Gamma\sum_{k=0}^{K}\sum_{r=1}^{k}\gamma^{k+1-r}\left(g_k^2+g_{r-1}^2\right)
			\nonumber+\sum_{k=0}^{K}g_k^2,\nonumber\\
			&\leq\Gamma\sum_{k=0}^{K}g_k^2\sum_{r=1}^{k}\gamma^{k+1-r}
			+\Gamma\sum_{k=0}^{K}\sum_{r=1}^{k}g_{r-1}^2\gamma^{k+1-r}+\sum_{k=0}^{K}g_k^2, \nonumber\\
			&\leq\frac{\Gamma}{1-\gamma}\sum_{k=0}^{K}g_k^2+\frac{\Gamma}{1-\gamma}\sum_{k=0}^{K}g_k^2
			+\sum_{k=0}^{K}g_k^2, \nonumber\\
			&\leq\left(1+\frac{2\Gamma}{1-\gamma}\right)\sum_{k=0}^{K}g_k^2. \nonumber
			\end{align}
			Therefore, we have 
			\begin{equation}
			\sum_{k=0}^{K}g_k\left\|\overline{\mb{z}}^{k+1}-\mb{x}_i^{k+1} \right\|
			\leq F\sum_{k=0}^{K}\alpha_k^2, \nonumber
			\end{equation}
			where $F=C\left(1+\frac{2\Gamma}{1-\gamma}\right)$. The last inequality is obtained from Lemma \ref{new1}.
		\end{proof}
	\end{lem}
\end{color}


\begin{color}{blue}
\begin{lem} \label{pre_optimal}
	Let Assumptions \ref{asp1}, \ref{asp2}, \ref{asp3} hold. Let~$\left\{\mb{z}_i^k\right\}$ be the sequence over~$k$ generated by Eq.~\eqref{alg2}. For $x^*\in \mc{X}^*$, we have the following:
	\begin{enumerate} [label=(\alph*)]
	\item The sequence $\left\{\left\|\overline{\mb{z}}^k-\mb{x}^*\right\|\right\}$ is convergent. 
	\item $\sum_{k=1}^{\infty}\alpha_k\left(f(\overline{\mb{z}}^{k})-f^*\right)<\infty.$
	\end{enumerate}
\end{lem}
\end{color}

\begin{proof}
	Consider Eq.~\eqref{alg2} and the fact that each column of~$M$ sums to one, we have the accumulation state
	\begin{align}
	\overline{\mb{z}}^{k+1}&=\overline{\mb{z}}^k+\frac{1}{n}\sum_{i=1}^{n}\mb{g}_i^k.\nonumber
	\end{align}
	Therefore, we obtain that
	\begin{align}
	&\left\|\overline{\mb{z}}^{k+1}-\mb{x}^*\right\|^2=\left\|\overline{\mb{z}}^k-\mb{x}^*\right\|^2+\left\|\frac{1}{n}\sum_{i=1}^{n}\mb{g}_i^k\right\|^2\nonumber\\
	&+\frac{2}{n}\sum_{i=1}^{n}\left\langle\overline{\mb{z}}^k-\mb{x}^*,\mb{g}_i^k\right\rangle,\nonumber\\
	&=\left\|\overline{\mb{z}}^k-\mb{x}^*\right\|^2+\frac{1}{n^2}\left\|\sum_{i=1}^{n}\mb{g}_i^k\right\|^2-\frac{2\alpha_k}{n}\sum_{i=1}^{n}\left\langle\overline{\mb{z}}^{k}-\mb{x}^{*},\nabla \mb{f}_i^k\right\rangle\nonumber\\
	&+\frac{2}{n}\sum_{i=1}^{n}\left\langle\overline{\mb{z}}^{k}-\mb{x}^{*},\mb{g}_i^k+\alpha_k\nabla \mb{f}_i^k\right\rangle.\label{mr_eq1}
	\end{align}
	Since $\|\nabla \mb{f}_i^k\|\leq B$, we have 
	\begin{align}
	&\left\langle\overline{\mb{z}}^{k}-\mb{x}^{*},\nabla \mb{f}_i^k\right\rangle=\left\langle\overline{\mb{z}}^{k}-\mb{x}_i^{k},\nabla \mb{f}_i^k\right\rangle+\left\langle\mb{x}_i^{k}-\mb{x}^{*},\nabla \mb{f}_i^k\right\rangle,\nonumber\\
	&\geq \left\langle\overline{\mb{z}}^{k}-\mb{x}_i^{k},\nabla \mb{f}_i^k\right\rangle+f_i(\mb{x}_i^k)-f_i(\mb{x}^*),\nonumber\\
	&\geq-B\left\|\overline{\mb{z}}^{k}-\mb{x}_i^{k}\right\|+f_i(\mb{x}_i^k)-f_i(\overline{\mb{z}}^{k})+f_i(\overline{\mb{z}}^{k})-f_i(\mb{x}^*),\nonumber\\
	&\geq-2B\left\|\overline{\mb{z}}^{k}-\mb{x}_i^{k}\right\|+f_i(\overline{\mb{z}}^{k})-f_i(\mb{x}^*).\label{mr_eq2}
	\end{align}
	By substituting Eq.~\eqref{mr_eq2} in Eq.~\eqref{mr_eq1}, we obtain that
	\begin{align}
	\frac{2\alpha_k}{n}\left(f(\overline{\mb{z}}^{k})-f^*\right)&\leq\left\|\overline{\mb{z}}^k-\mb{x}^*\right\|^2-\left\|\overline{\mb{z}}^{k+1}-\mb{x}^*\right\|^2\nonumber\\
	&+\frac{1}{n^2}\left\|\sum_{i=1}^{n}\mb{g}_i^k\right\|^2+\frac{4B\alpha_k}{n}\sum_{i=1}^n\left\|\overline{\mb{z}}^{k}-\mb{x}_i^{k}\right\|\nonumber\\
	&+\frac{2}{n}\sum_{i=1}^{n}\left\langle\overline{\mb{z}}^{k}-\mb{x}^{*},\mb{g}_i^k+\alpha_k\nabla \mb{f}_i^k\right\rangle.\label{mr_eq3}
	\end{align}
	We now analyze the last term in Eq.~\eqref{mr_eq3}.
	\begin{align}
	&\sum_{i=1}^{n}\left\langle\overline{\mb{z}}^k-\mb{x}^*,\mb{g}_i^k+\alpha_k\nabla\mb{f}_i^k\right\rangle=\sum_{i=1}^{n}\left\langle\overline{\mb{z}}^k-\overline{\mb{z}}^{k+1},\mb{g}_i^k+\alpha_k\nabla \mb{f}_i^k\right\rangle\nonumber\\
	&+\sum_{i=1}^{n}\left\langle\overline{\mb{z}}^{k+1}-\mb{x}_i^{k+1},\mb{g}_i^k+\alpha_k\nabla \mb{f}_i^k\right\rangle\nonumber\\
	&+\sum_{i=1}^{n}\left\langle\mb{x}_i^{k+1}-\mb{x}^*,\mb{g}_i^k+\alpha_k\nabla \mb{f}_i^k\right\rangle\nonumber\\
	&:=s_1+s_2+s_3\label{mr_eq4}
	\end{align}
	where~$s_1$,~$s_2$, and~$s_3$ denote each of RHS terms in Eq.~\eqref{mr_eq4}. We discuss each term in sequence. 
\begin{color}{blue}
		Since $\|\nabla \mb{f}_i^k\|\leq B$, we have:
		\begin{align}
		s_1&=-\sum_{i=1}^n\left\langle\mb{g}_i^k,\mb{g}_i^k+\alpha_k\nabla \mb{f}_i^k\right\rangle
		\leq B\alpha_k\sum_{i=1}^n\left\|\mb{g}_i^k\right\|=B\alpha_kg_k; \nonumber\\
		s_2&\leq\sum_{i=1}^n\left\|\overline{\mb{z}}^{k+1}-\mb{x}_i^{k+1}\right\|\left\|\mb{g}_i^k+\alpha_k\nabla\mb{f}_i^k\right\|, \nonumber\\
		&\leq\sum_{i=1}^n\left\|\overline{\mb{z}}^{k+1}-\mb{x}_i^{k+1}\right\|\left(
		\left\|\mb{g}_i^k\right\|+\alpha_k\left\|\nabla\mb{f}_i^k\right\|      \right)
		\nonumber,\\
		&\leq\sum_{i=1}^n\left\|\overline{\mb{z}}^{k+1}-\mb{x}_i^{k+1}\right\|\left(g_k+B\alpha_k\right). \nonumber
		\end{align}
        Using the result of Lemma~\ref{lem_NonexpanBregman}(a), we have for any $i$
		\begin{align}
		\left\langle\mb{x}_i^{k+1}-\mb{x}^*,\mb{g}_i^k+\alpha_k\nabla \mb{f}_i^k\right\rangle\leq0,\nonumber
		\end{align}
		which reveals that $s_3\leq0$.\\
		Using the bounds on $s_1$,$s_2$,$s_3$ and Lemma~\ref{new1} and Lemma~\ref{new2}, we can derive the following:
		\begin{align}
		\frac{2\alpha_k}{n}\left(f(\overline{\mb{z}}^{k})-f^*\right) 
		\leq& \left\|\overline{\mb{z}}^{k}-\mb{x}^*\right\|^2-\left\|\overline{\mb{z}}^{k+1}-\mb{x}^*\right\|^2
		\nonumber \\
		&+
		\frac{C^2}{n^2}\alpha_k^2+
		\frac{4B\alpha_k}{n}\sum_{i=1}^n\left\|\overline{\mb{z}}^{k}-\mb{x}_i^{k}\right\| \nonumber\\
		&+\frac{2BC}{n}\alpha_k^2 \nonumber\\
		&+\frac{2(B+C)}{n}\alpha_k\sum_{i=1}^n\left\|\overline{\mb{z}}^{k+1}-\mb{x}_i^{k+1}\right\|. \label{fconvergence}
		\end{align}
		{\color{blue}
		Let
        \begin{align}
        	h_k =& \frac{C^2}{n^2}\alpha_k^2+
        	\frac{4B\alpha_k}{n}\sum_{i=1}^n\left\|\overline{\mb{z}}^{k}-\mb{x}_i^{k}\right\| \nonumber
        	+\frac{2BC}{n}\alpha_k^2 \nonumber\\
        	&+\frac{2(B+C)}{n}\alpha_k\sum_{i=1}^n\left\|\overline{\mb{z}}^{k+1}-\mb{x}_i^{k+1}\right\|. \nonumber
        \end{align}
		Since the step-size $\alpha_k$ satisfies $\sum_{k=0}^{\infty}\alpha_k^2<\infty$, together with Lemma \ref{new},
    	we have $\sum_{k=0}^{\infty}h_k<\infty$. Therefore,
    	\begin{align}
    		\sum_{k=0}^{\infty}\alpha_k\left(f(\overline{\mb{z}}^{k})-f^*\right)
    		\leq n\left\|\overline{\mb{z}}^{0}-\mb{x}^*\right\|^2+n\sum_{k=0}^{\infty}h_k<\infty \nonumber,
    	\end{align}
    	which completes the second part of the proof.\\
    	By rearranging equation Eq.~\eqref{fconvergence}, we have:
    	\begin{align}
    		\left\|\overline{\mb{z}}^{k+1}-\mb{x}^*\right\|^2
    		\leq \left\|\overline{\mb{z}}^k-\mb{x}^*\right\|^2-\frac{2\alpha_k}{n}\left(f(\overline{\mb{z}}^{k})-f^*\right)
    		+h_k. \nonumber
    	\end{align}
    	Since $\frac{2\alpha_k}{n}\left(f(\overline{\mb{z}}^{k})-f^*\right)\geq0$, 
    	\begin{equation}
    		\left\|\overline{\mb{z}}^{k+1}-\mb{x}^*\right\|^2
    		\leq \left\|\overline{\mb{z}}^k-\mb{x}^*\right\|^2 + h_k. \nonumber
    	\end{equation}
    	Let $r_k = \left\|\overline{\mb{z}}^k-\mb{x}^*\right\|^2 + \sum_{s=k}^{\infty}h_s.$ It follows that
    	\begin{align}
    		r_{k+1}
    		=&\left\|\overline{\mb{z}}^{k+1}-\mb{x}^*\right\|^2 + \sum_{s=k+1}^{\infty}h_s, \nonumber \\
    		\leq& \left\|\overline{\mb{z}}^k-\mb{x}^*\right\|^2 + h_k + \sum_{s=k+1}^{\infty}h_s, \nonumber \\
    		=&r_k, \nonumber
    	\end{align}
    	which leads to the fact that $\left\{r_k\right\}$ is a non-increasing, nonnegative sequence. Therefore, $\left\{r_k\right\}$ converges to some nonnegative point. Since $\lim\limits_{k\to\infty}\sum_{s=k}^{\infty}h_s=0$,
    	\begin{align}
    		\lim\limits_{k\to\infty}\left\|\overline{\mb{z}}^k-\mb{x}^*\right\|^2
    		= \lim\limits_{k\to\infty}(r_k-\sum_{s=k}^{\infty}h_s) = \lim\limits_{k\to\infty}r_k. \nonumber
    	\end{align}  
    	Therefore, The sequence $\left\{\left\|\overline{\mb{z}}^k-\mb{x}^*\right\|\right\}$ is convergent.
        }
\end{color}
\end{proof}

\begin{color}{blue}
	\begin{theorem}
	Let Assumptions \ref{asp1}, \ref{asp2}, \ref{asp3} hold. Let~$\left\{\mb{x}_i^k\right\}$ be the sequence over~$k$ generated by Eq.~\eqref{alg2}. For $\forall i \in \mc{V}$, 
    \begin{equation}
    	{\color{blue}\lim\limits_{k\to\infty}\mb{x}_i^k=\mb{x}^*,\nonumber}
    \end{equation}where $\mb{x}^*\in \mc{X}^*$.
	\end{theorem}
    \begin{proof}
    	According to Lemma \ref{pre_optimal}, the sequence $\left\{\left\|\overline{\mb{z}}^k-\mb{x}^*\right\|\right\}$ is convergent and $\sum_{k=1}^{\infty}\alpha_k\left(f(\overline{\mb{z}}^{k})-f^*\right)<\infty$. Since $f(\overline{\mb{z}}^{k})\geq f^*$, $\liminf_{k\to\infty}\left(f(\overline{\mb{z}}^{k})-f^*\right) \geq0.$ We claim that for $\forall \epsilon>0$, there are infinite terms in the sequence $\left\{f(\overline{\mb{z}}^{k})-f^*\right\}$ such that $f(\overline{\mb{z}}^{k})-f^* < \epsilon$ and therefore $\liminf_{k\to\infty}\left(f(\overline{\mb{z}}^{k})-f^*\right) \leq0$. Otherwise, there exists some integer $K_1$ such that $f(\overline{\mb{z}}^{k})-f^* \geq \epsilon$ for all $k>K_1$. Then we have the following:
    	\begin{eqnarray*}
    	&&\sum_{k=1}^{\infty}\alpha_k(f(\overline{\mb{z}}^{k})-f^*) \\
    	&=& \sum_{k=1}^{K_1}\alpha_k(f(\overline{\mb{z}}^{k})-f^*) 
    	+\sum_{k=K_1+1}^{\infty}\alpha_k(f(\overline{\mb{z}}^{k})-f^*)\\ 
    	&\geq& \sum_{k=1}^{K_1}\alpha_k(f(\overline{\mb{z}}^{k})-f^*) 
    	+\epsilon\sum_{k=K_1+1}^{\infty}\alpha_k
    	> \infty, 
    	\end{eqnarray*}
    	which is a contradiction. Hence, $$\liminf_{k\to\infty}\left(f(\overline{\mb{z}}^{k})-f^*\right) =0.$$ 
    	{\color{blue}Then there exists a subsequence of $\left\{f(\overline{\mb{z}}^{k})\right\}$, $\left\{{f(\overline{\mb{z}}^{k_l}})\right\}$}such that $\lim_{l\to\infty}{f(\overline{\mb{z}}^{k_l}})=f^*.$ Since $\left\{\overline{\mb{z}}^k\right\}$ is a bounded sequence, we assume without loss of generality that $\lim_{l\to\infty}{\overline{\mb{z}}^{k_l}}=\mb{y}$, where $\mb{y}\in\mc{X}$ (Otherwise we can select a convergent subsequence of $\left\{{\overline{\mb{z}}^{k_l}}\right\}$). Due to the continuity of $f$ over its domain, $\liminf_{l\to\infty}f({\overline{\mb{z}}^{k_l}})=f(\mb{y}).$ Therefore, we have $f(\mb{y})=f^*$ and $\mb{y}\in \mc{X^*}$ due to the uniqueness of the limit point of a sequence. Let $x^*=\mb{y}$. Since $\lim_{l\to\infty}\left\|{\overline{\mb{z}}^{k_l}}-\mb{x}^*\right\|=0$ and $\left\{\left\|\overline{\mb{z}}^k-\mb{x}^*\right\|\right\}$ is convergent, we have $\lim_{k\to\infty}\left\|\overline{\mb{z}}^k-\mb{x}^*\right\|=0$. 
        {\color{blue}
        	Then it follows that:~$\forall $$i\in\mc{V}$,
            \begin{align}
            	\left\|\mb{x}_i^k-\mb{x}^*\right\|
            	=& \left\|\mb{x}_i^k-\overline{\mb{z}}^k+\overline{\mb{z}}^k-\mb{x}^*\right\| \nonumber \\
            	\leq& \left\|\mb{x}_i^k-\overline{\mb{z}}^k\right\|+\left\|\overline{\mb{z}}^k-\mb{x}^*\right\|. \nonumber
            \end{align}
            Therefore, according to previous discussion and Lemma \ref{lem_consensus2}, we have:  
            \begin{equation}
            	\lim\limits_{k\to\infty}\left\|\mb{x}_i^k-\mb{x}^*\right\|=0, \nonumber
            \end{equation}
            which completes the proof.
        }
    	
    \end{proof}
\end{color}

\subsection{Convergence Rate}
Let~$f_K^*:=\min_{0<k\leq K}f(\overline{\mb{z}}^k)$. we have
\begin{align}\label{rate_ineq}
(f_K^*-f^*)\sum_{k=1}^K\alpha_k\leq\sum_{k=1}^K\alpha_k(f(\overline{\mb{z}}^k)-f^*).
\end{align}
By combining Eqs.~\eqref{alphaconsensus} and~\eqref{fconvergence}, Eq.~\eqref{rate_ineq} leads to
\begin{align}
(f_K^*-f^*)\sum_{k=1}^K\alpha_k\leq C_1+C_2\sum_{k=1}^K\alpha_k^2,\nonumber
\end{align}
or equivalently,
\begin{align}\label{rate}
(f_K^*-f^*)\leq\frac{ C_1}{\sum_{k=1}^K\alpha_k}+\frac{C_2\sum_{k=1}^K\alpha_k^2}{\sum_{k=1}^K\alpha_k},
\end{align}
where the constants,~$C_1$ and~$C_2$, are given by
{\small
\begin{equation*}
C_1=\frac{n}{2}\left\|\overline{\mb{z}}^0-\mb{x}^*\right\|^2, \nonumber
\end{equation*}
\begin{color}{blue}
	\begin{equation}
		C_2=\frac{C}{2n}+BD+Fn+3nBD\left(1+\frac{\Gamma}{1-\gamma}\right). \nonumber
	\end{equation}
\end{color}
}
\begin{color}{blue}
We choose the step-size of $\alpha_k=k^{-1/2}$ and use the inequalities as follows:
	\begin{equation}
		\sum_{k=1}^{K}\frac{1}{k} < \ln K +1, \nonumber 
	\end{equation}
	\begin{equation}
		\sum_{k=1}^{K}\frac{1}{\sqrt{k}} > 2(\sqrt{K+1}-1). \nonumber 
	\end{equation}
\end{color} 
The first term in Eq.~\eqref{rate}  leads to
\begin{align}
\frac{ C_1}{\sum_{k=1}^K\alpha_k}<C_1\frac{1/2}{{\color{blue}\sqrt{K+1}}-1}=O\left(\frac{1}{\sqrt{K}}\right),\nonumber
\end{align}
while the second term in Eq.~\eqref{rate} leads to
\begin{align}
\frac{C_2\sum_{k=1}^K\alpha_k^2}{\sum_{k=1}^K\alpha_k} < C_2\frac{
	1+\ln K}{2({\color{blue}\sqrt{K+1}}-1)}=O\left(\frac{\ln K}{\sqrt{K}}\right).\nonumber
\end{align}
It can be observed that $O\left(\frac{\ln K}{\sqrt{K}}\right)$ dominates.

In conclusion, {\color{blue}we achieve the convergence rate of $O(\frac{\ln k}{\sqrt{k}})$ by choosing the step-size of $\frac{1}{\sqrt{k}}$.} This convergence rate is the same as the distributed projected subgradient method,~\cite{cc_nedic}, solving constrained optimization over undirected graphs. Therefore, the restriction of directed graphs does not affect the convergence speed. 
\section{Numerical Results}\label{s4}
Consider the application of D-DPS for solving a distributed logistic regression problem over a directed graph:
\begin{align}
\mb{x}^*=\underset{\mb{x}\in\mc{X}\subset\mbb{R}^p}{\operatorname{argmin}}\sum_{i=1}^n\sum_{j=1}^{m_i}\ln\left[1+\exp\left(-\left(\mb{c}_{ij}^\top\mb{x}\right)y_{ij}\right)\right],\nonumber
\end{align}
where $\mc{X}$ is a small convex set restricting the value of $\mb{x}$ to avoid overfitting. Each agent $i$ has access to $m_i$ training samples, $(\mb{c}_{ij},y_{ij})\in\mbb{R}^p\times\{-1,+1\}$, where $\mb{c}_{ij}$ includes the~$p$ features of the $j$th training sample of agent $i$, and $y_{ij}$ is the corresponding label. This problem can be formulated in the form of P1 with the private objective function $f_i$ being
\begin{align}
f_i(\mb{x})=\sum_{j=1}^{m_i}\ln\left[1+\exp\left(-\left(\mb{c}_{ij}^\top\mb{x}\right)y_{ij}\right)\right],\quad\mbox{s.t. }\mb{x}\in\mc{X}.\nonumber
\end{align}
In our setting, we have $n=10$, $m_i=10$, for all $i$, and $p=100$. The constrained set is described by a ball in $\mbb{R}^p$. We consider the network topology as the digraph shown in Fig.~\ref{graph}.
\begin{figure}[!h]
	\begin{center}
		\noindent
		\includegraphics[width=1.55in]{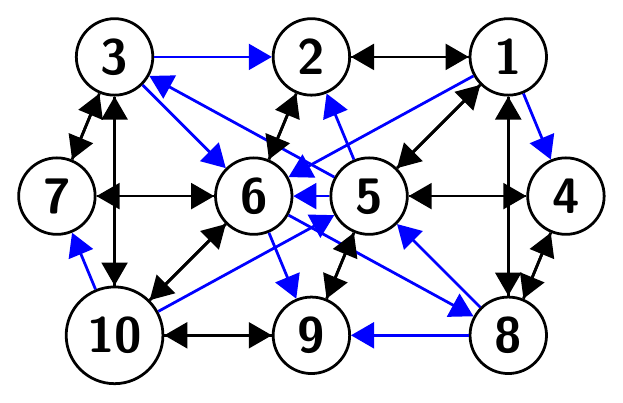}
		\caption{A strongly-connected but non-balanced directed graph.}\label{graph}
	\end{center}
\end{figure}
We plot the residuals~$\frac{\left\|\mb{x}_i^k-\mb{x}^*\right\|_F}{\left\|\mb{x}_i^0-\mb{x}^*\right\|_F}$ for each agent $i$ as a function of $k$ in Fig.~\ref{resid} (Left). In Fig.~\ref{resid} (Right), we show the disagreement between the state estimate of each agent and the accumulation state, and the additional variables of all agents. The experiment follows the results of Lemma \ref{lem_consensus2} that both the disagreements and the additional variables converge to zero. 

We compare the convergence of D-DPS with others related algorithms, Subgradient-Push (SP),~\cite{opdirect_Nedic}, and Weight-Balancing Subgradient Descent (WBSD),~\cite{opdirect_Makhdoumi}, in Fig.~\ref{diff_alg}. Since both SP and WBSD are algorithms for unconstrained problems, we reformulate the problem in an approximate form, 
\begin{align}
f_i(\mb{x})=\lambda\|\mb{x}\|^2+\sum_{j=1}^{m_i}\ln\left[1+\exp\left(-\left(\mb{c}_{ij}^\top\mb{x}\right)y_{ij}\right)\right],\nonumber
\end{align}
where the regularization term $\lambda\|\mb{x}\|^2$ is an approximation to replace the original constrained set to avoid overfitting. It can be observed from Fig.~\ref{diff_alg} that all three algorithms have the same order of convergence rate. However, D-DPS is further suited for the constrained problems.
\begin{figure}[!h]
\centering
\subfigure{\includegraphics[width=1.71in]{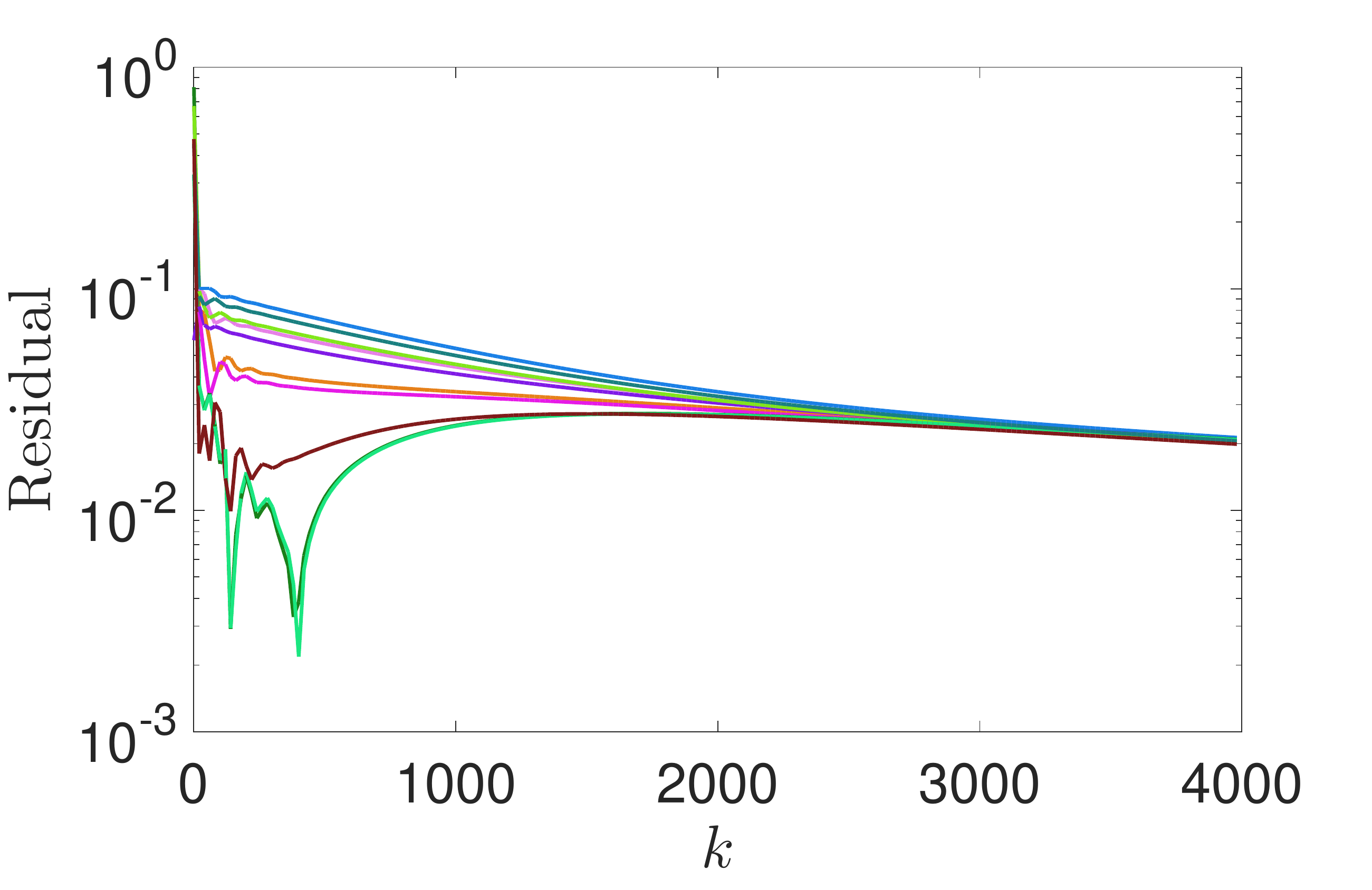}}
\subfigure{\includegraphics[width=1.73in]{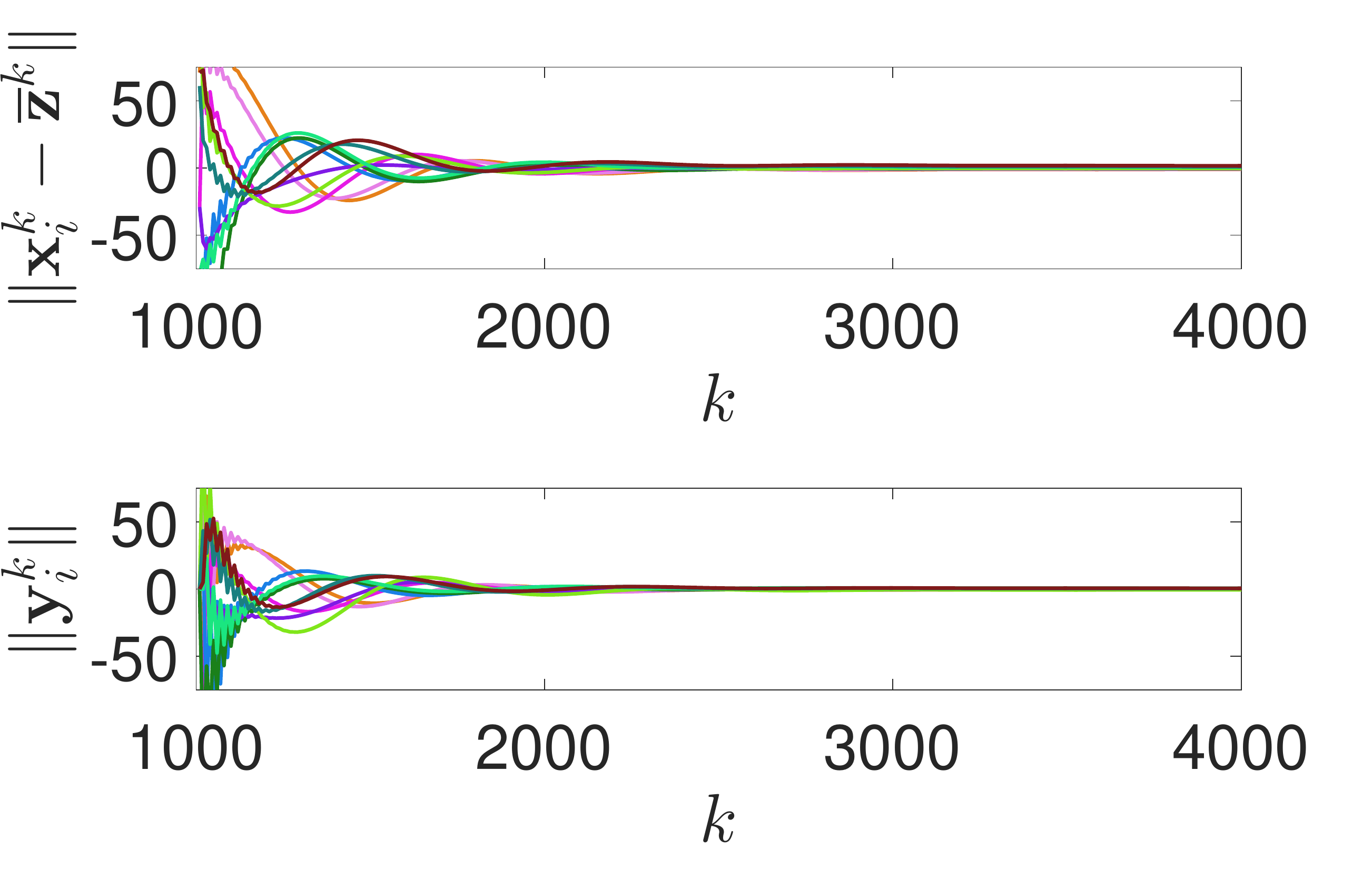}}
\caption{(Left) D-DPS residuals at $10$ agents. (Right) Sample paths of states,~$\|\mb{x}_i^k-\overline{\mb{z}}^k\|$, and~$\|\mb{y}_i^k\|$, for all agents.}
\label{resid}
\end{figure}
\begin{figure}[!h]
	\begin{center}
		\noindent
		\includegraphics[width=2.8in]{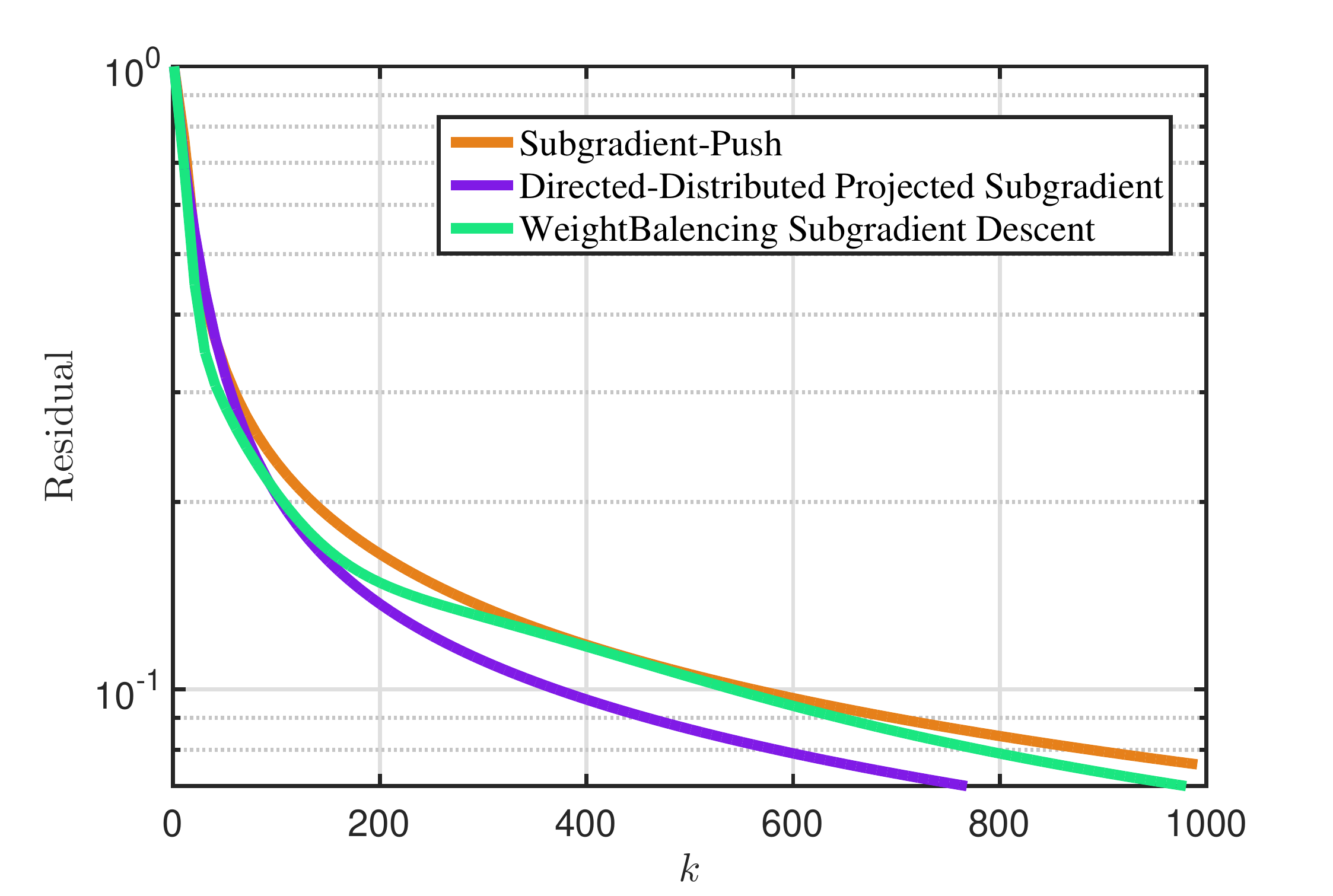}
		\caption{Convergence comparison between different algorithms.}\label{diff_alg}
	\end{center}
\end{figure}

\section{Conclusions}\label{s5}
In this paper, we present a distributed solution, D-DPS, to the \emph{constrained} optimization problem over \emph{directed} multi-agent networks, where the agents' goal is to collectively minimize the sum of locally known convex functions. Compared to the algorithm solving over undirected networks, the D-DPS simultaneously constructs a row-stochastic matrix and a column-stochastic matrix instead of only a doubly-stochastic matrix. This enables all agents to overcome the asymmetry caused by the directed communication network. We show that D-DPS converges to the optimal solution and the convergence rate is $O(\frac{\ln k}{\sqrt{k}})$, where $k$ is the number of iterations. In future, we will consider solving the distributed constrained optimization problems over directed and time-varying graph under, possibly, asynchronous information exchange. 

{
\footnotesize
\bibliographystyle{IEEEbib}
\bibliography{ref}

\begin{thebibliography}{10}

\bibitem{c_Hooi-Tong}
L.~Hooi-Tong,
\newblock ``On a class of directed graphs{-}with an application to traffic-flow
  problems,''
\newblock {\em Operations Research}, vol. 18, no. 1, pp. 87--94, 1970.

\bibitem{distributed_Rabbit}
M.~Rabbat and R.~Nowak,
\newblock ``Distributed optimization in sensor networks,''
\newblock in {\em 3rd International Symposium on Information Processing in
  Sensor Networks}, Berkeley, CA, Apr. 2004, pp. 20--27.

\bibitem{distributed_Cevher}
V.~Cevher, S.~Becker, and M.~Schmidt,
\newblock ``Convex optimization for big data: Scalable, randomized, and
  parallel algorithms for big data analytics,''
\newblock {\em IEEE Signal Processing Magazine}, vol. 31, no. 5, pp. 32--43,
  2014.

\bibitem{distributed_Mateos}
G.~Mateos, J.~A. Bazerque, and G.~B. Giannakis,
\newblock ``Distributed sparse linear regression,''
\newblock {\em IEEE Transactions on Signal Processing}, vol. 58, no. 10, pp.
  5262--5276, Oct. 2010.

\bibitem{distributed_Ling}
Q.~Ling, Y~Xu, W~Yin, and Z~Wen,
\newblock ``Decentralized low-rank matrix completion,''
\newblock in {\em 2012 IEEE International Conference on Acoustics, Speech and
  Signal Processing}, Mar. 2012, pp. 2925--2928.

\bibitem{uc_Nedic}
A.~Nedic and A.~Ozdaglar,
\newblock ``Distributed subgradient methods for multi-agent optimization,''
\newblock {\em IEEE Transactions on Automatic Control}, vol. 54, no. 1, pp.
  48--61, Jan. 2009.

\bibitem{cc_nedic}
A.~Nedic, A.~Ozdaglar, and P.~A. Parrilo,
\newblock ``Constrained consensus and optimization in multi-agent networks,''
\newblock {\em IEEE Transactions on Automatic Control}, vol. 55, no. 4, pp.
  922--938, Apr. 2010.

\bibitem{cc_Duchi}
J.~C. Duchi, A.~Agarwal, and M.~J. Wainwright,
\newblock ``Dual averaging for distributed optimization: Convergence analysis
  and network scaling,''
\newblock {\em IEEE Transactions on Automatic Control}, vol. 57, no. 3, pp.
  592--606, Mar. 2012.

\bibitem{cc_Johansson}
B.~Johansson, T.~Keviczky, M.~Johansson, and K.~H. Johansson,
\newblock ``Subgradient methods and consensus algorithms for solving convex
  optimization problems,''
\newblock in {\em 47th IEEE Conference on Decision and Control}, Cancun,
  Mexico, Dec. 2008, pp. 4185--4190.

\bibitem{srivastava2011distributed}
K.~Srivastava and A.~Nedic,
\newblock ``Distributed asynchronous constrained stochastic optimization,''
\newblock {\em IEEE Journal of Selected Topics in Signal Processing}, vol. 5,
  no. 4, pp. 772--790, 2011.

\bibitem{ram2010distributed}
S.~S. Ram, A.~Nedi{\'c}, and V.~V. Veeravalli,
\newblock ``Distributed stochastic subgradient projection algorithms for convex
  optimization,''
\newblock {\em Journal of optimization theory and applications}, vol. 147, no.
  3, pp. 516--545, 2010.

\bibitem{fast_Gradient}
D.~Jakovetic, J.~Xavier, and J.~M.~F. Moura,
\newblock ``Fast distributed gradient methods,''
\newblock {\em IEEE Transactions on Automatic Control}, vol. 59, no. 5, pp.
  1131--1146, 2014.

\bibitem{EXTRA}
W.~Shi, Q.~Ling, G.~Wu, and W~Yin,
\newblock ``{EXTRA: A}n exact first-order algorithm for decentralized consensus
  optimization,''
\newblock {\em SIAM Journal on Optimization}, vol. 25, no. 2, pp. 944--966,
  2015.

\bibitem{dual_Terelius}
H.~Terelius, U.~Topcu, and R.~M. Murray,
\newblock ``Decentralized multi-agent optimization via dual decomposition,''
\newblock {\em IFAC Proceedings Volumes}, vol. 44, no. 1, pp. 11245--11251,
  2011.

\bibitem{ADMM_Mota}
J.~F.~C. Mota, J.~M.~F. Xavier, P.~M.~Q. Aguiar, and M.~Puschel,
\newblock ``D-{ADMM}: A communication-efficient distributed algorithm for
  separable optimization,''
\newblock {\em IEEE Transactions on Signal Processing}, vol. 61, no. 10, pp.
  2718--2723, May 2013.

\bibitem{ADMM_Shi}
W.~Shi, Q.~Ling, K~Yuan, G~Wu, and W~Yin,
\newblock ``On the linear convergence of the {ADMM} in decentralized consensus
  optimization,''
\newblock {\em IEEE Transactions on Signal Processing}, vol. 62, no. 7, pp.
  1750--1761, April 2014.

\bibitem{ADMM_Wei}
E.~Wei and A.~Ozdaglar,
\newblock ``Distributed alternating direction method of multipliers,''
\newblock in {\em 51st IEEE Annual Conference on Decision and Control}, Dec.
  2012, pp. 5445--5450.

\bibitem{ADMM_jakovetic}
D.~Jakoveti{\'c}, J.~M .~F. Moura, and J.~Xavier,
\newblock ``Linear convergence rate of a class of distributed augmented
  lagrangian algorithms,''
\newblock {\em IEEE Transactions on Automatic Control}, vol. 60, no. 4, pp.
  922--936, 2015.

\bibitem{bianchi2014stochastic}
P.~Bianchi, W.~Hachem, and F.~Iutzeler,
\newblock ``A stochastic coordinate descent primal-dual algorithm and
  applications to large-scale composite optimization,''
\newblock {\em arXiv preprint arXiv:1407.0898}, 2014.

\bibitem{ADMM_Hong}
M.~Hong and Z.~Luo,
\newblock ``On the linear convergence of the alternating direction method of
  multipliers,''
\newblock {\em arXiv preprint arXiv:1208.3922}, 2012.

\bibitem{ADMM_Ling}
Q.~Ling, W.~Shi, G.~Wu, and A.~Ribeiro,
\newblock ``{DLM: D}ecentralized linearized alternating direction method of
  multipliers,''
\newblock {\em IEEE Transactions on Signal Processing}, vol. 63, no. 15, pp.
  4051--4064, Aug. 2015.

\bibitem{Necoara_CDA0}
I.~Necoara, Yu. Nesterov, and F.~Glineur,
\newblock ``Random block coordinate descent methods for linearly constrained
  optimization over networks,''
\newblock {\em Tech. Rep., UPB}, 2014.

\bibitem{Necoara_CDA1}
I.~Necoara,
\newblock ``Random coordinate descent algorithms for multi-agent convex
  optimization over networks,''
\newblock {\em IEEE Transactions on Automatic Control}, vol. 58, no. 8, pp.
  2001--2012, Aug. 2013.

\bibitem{opdirect_Nedic}
A.~Nedic and A.~Olshevsky,
\newblock ``Distributed optimization over time-varying directed graphs,''
\newblock {\em IEEE Transactions on Automatic Control}, vol. PP, no. 99, pp.
  1--1, 2014.

\bibitem{opdirect_Xi}
C.~Xi, Q.~Wu, and U.~A. Khan,
\newblock ``On the distributed optimization over directed networks,''
\newblock {\em arXiv preprint arXiv:1510.02146}, 2015.

\bibitem{opdirect_Makhdoumi}
A.~Makhdoumi and A.~Ozdaglar,
\newblock ``Graph balancing for distributed subgradient methods over directed
  graphs,''
\newblock {\em to appear in 54th IEEE Annual Conference on Decision and
  Control}, 2015.

\bibitem{opdirect_Xi2}
C.~Xi and U.~A. Khan,
\newblock ``On the linear convergence of distributed optimization over directed
  graphs,''
\newblock {\em arXiv preprint arXiv:1510.02149}, 2015.

\bibitem{ac_directed}
F.~Benezit, V.~Blondel, P.~Thiran, J.~Tsitsiklis, and M.~Vetterli,
\newblock ``Weighted gossip: Distributed averaging using non-doubly stochastic
  matrices,''
\newblock in {\em IEEE International Symposium on Information Theory}, Jun.
  2010, pp. 1753--1757.

\bibitem{ac_directed0}
D.~Kempe, A.~Dobra, and J.~Gehrke,
\newblock ``Gossip-based computation of aggregate information,''
\newblock in {\em 44th Annual IEEE Symposium on Foundations of Computer
  Science}, Oct. 2003, pp. 482--491.

\bibitem{ac_Cai1}
K.~Cai and H.~Ishii,
\newblock ``Average consensus on general strongly connected digraphs,''
\newblock {\em Automatica}, vol. 48, no. 11, pp. 2750 -- 2761, 2012.

\end{thebibliography}
}

\end{document}